\documentclass{amsart}
\usepackage{amsmath,amssymb,hyperref,mathrsfs,graphicx}
\usepackage[utf8x]{inputenc}
\usepackage{amsmath}
\usepackage{amsfonts}
\usepackage{latexsym}
\usepackage{amsthm}
\usepackage{amssymb,amscd}
\usepackage{xargs}
\usepackage{tikz}
\usepackage{graphicx}
\usepackage{color}
\usepackage{enumitem}
\usepackage[colorinlistoftodos,prependcaption]{todonotes}
\presetkeys%
{todonotes}%
{inline,backgroundcolor=white}{}
\tikzset{/tikz/notestyleraw/.append style={text=black}}

\newtheorem{thm}{Theorem}[section]

\newtheorem{lem}[thm]{Lemma}
\newtheorem{defn}[thm]{Definition}
\newtheorem{prop}[thm]{Proposition}

\newtheorem{rmk}[thm]{Remark}
\newcommand{\be}{\begin{eqnarray}}
\newcommand{\ee}{\end{eqnarray}}
\newcommand{\ben}{\begin{eqnarray*}}
\newcommand{\een}{\end{eqnarray*}}
\newcommand{\beq}{\begin{equation}}
\newcommand{\eeq}{\end{equation}}
\newcommand{\beal}{\begin{aligned}}
\newcommand{\enal}{\end{aligned}}

\newcommand{\eps}{\varepsilon}

\newcommand{\lb}{\lambda}

\newcommand{\T}{\mathbb{T}}
\newcommand{\R}{\mathbb{R}}

\newcommand{\N}{\mathbb{N}}

\newcommand{\Z}{\mathbb{Z}}

\newcommand{\om}{\omega}
\newcommand{\dt}{\delta}
\newcommand{\Dt}{\Delta}

\newcommand{\cS}{\mathcal{S}}

\newcommand{\cM}{\mathcal{M}}

\newcommand{\cP}{\mathcal{P}}

\newcommand{\cL}{\mathcal{L}}

\newcommand{\ol}{\overline }

\title[Limit of solutions for semilinear H-J equations with degenerate viscosity]{Limit of solutions for semilinear Hamilton-Jacobi equations with degenerate viscosity}
\thanks{*\,The datasets analysed during the current study are available from the corresponding author on reasonable request}
\date{}
\begin{document}
\maketitle

\centerline{\scshape Jianlu Zhang*}
\medskip
{\footnotesize
\centerline{Hua Loo-Keng Key Laboratory of Mathematics \&}
 \centerline{Mathematics Institute, Academy of Mathematics and systems science}
 \centerline{Chinese Academy of Sciences, Beijing 100190, China}
  \centerline{{\it Email: }jellychung1987@gmail.com}  
}
\bigskip

\begin{abstract}
In the paper we prove the convergence of viscosity solutions $u_{\lb}$ as $\lb\rightarrow0_+$ for the parametrized degenerate viscous Hamilton-Jacobi equation
\[
H(x,d_x u, \lb u)=\alpha(x)\Dt u,\quad \alpha(x)\geq 0,\quad x\in \T^n
\]
under suitable convex and monotonic conditions on $H:  T^*M\times\R\rightarrow\R$. Such a limit can be characterized in terms of stochastic Mather measures associated with the critical equation 
\[
H(x,d_x u,0)=\alpha(x)\Dt u.
\]


\end{abstract}

\bigskip

\section{Introduction}

Let $\T^n:=\R^n\slash 2\pi\Z^n$ be the $n-$dimensional torus equipped with the Euclid metric $|\cdot|$. The {\it Hamiltonian} $H : T^*\T^n \times\R \rightarrow \R$ is a $C^{2}-$function which satisfies the following assumptions:

\begin{itemize}
\item[(H1)] For any $(x,u)\in\T^n\times\R$, $H(x,\cdot,u)$  is strictly convex with respect to  $p\in T_x^*\T^n$;
\item[(H2)] There exists constants $m>1$ and $K_m, M_m>0$ such that 
\[
H(x,p,0)\geq K_m|p|^m-M_m,\quad\forall (x,p)\in T^*\T^n;
\]
\item[(H3)] There exist $0<\rho_*\leq\rho^*\in\R$ such that for any $(x,p)\in T^*\T^n$ and $u_1\leq u_2$, 
\[
\rho_*(u_2-u_1)\leq H(x,p,u_2)-H(x,p,u_1)\leq \rho^*(u_2-u_1);
\]
\item[(H4)] For any $(x,p),(y,p)\in T^*\T^n$ and $|u|\leq R$,  there exist constants $\kappa(R),\varsigma(R)>0$ such that
\[
|H(x,p,u)-H(y,p,u)|\leq \kappa(R)\Big(H(x,p,0)+\varsigma(R)\Big)|x-y|;
\]
\item[(H5)] For any $x\in\T^n$, $p,p'\in T_x^*\T^n$ with $|p'|\leq 2|p|$ and $|u|\leq R$, there exist  constants $\xi(R),\eta(R)>0$ such that
\[
|H(x,p,u)-H(x,p',u)|\leq\xi(R)\Big(H(x,p,u)+\eta(R)\Big)\frac{ |p-p'|}{|p|+1}.
\]
\end{itemize}


Following previous assumptions {\rm (H1)-(H5)}, in this paper we study the asymptotic limit of the viscosity solution $u_\lb$ as $\lb\rightarrow 0_+$ for the following semilinear Hamilton-Jacobi equations with a degenerate diffusion:
\beq\label{eq:hj-e}\tag{HJ$_{e}^\lb$}
H(x,d_xu_\lb,\lb u_\lb)=\alpha(x)\Delta u_\lb+c(H),\quad x\in \T^n
\eeq
where $C^2(\T^n,\R)\ni \alpha(x)\geq0$ and the {\it ergodic constant} $c(H)\in\R$ is suitably chosen such that the critical equation 
\beq\label{eq:hj-e-cri}\tag{HJ$_{e}^{0}$}
H(x,d_xu,0)=\alpha(x)\Delta u+c(H),\quad x\in \T^n
\eeq
is solvable. Notice that  $u_\lb$ is unique due to a {\it comparison principle} (see \cite{CIL} for instance), and the uniqueness of $c(H)$ has been established in \cite{MT,IMT1,IMT2}. However, the degeneracy of $\alpha(x)$ disable the uniqueness of solutions for \eqref{eq:hj-e-cri}, even up to an additive constant. So the convergence of $u_\lb$ as $\lb\rightarrow 0_+$ is uncertain and need to be proved.

In this paper, we verify the convergence of $u_\lb$ by presenting the following conclusion. Without loss of generality, we assume $ c(H) = 0$ henceforth.
\begin{thm}\label{thm:1}
Under the assumptions {\rm (H1)-(H5)}, the viscosity solution $u_\lb$ of \eqref{eq:hj-e} converges to a uniquely identified solution $u_0$ of \eqref{eq:hj-e-cri} as $\lb\rightarrow 0_+$, which 
can be expressed by 
\[
u_0(x)=\sup_{\om\in\cS'} \om(x),\quad x\in \T^n
\]
with $\cS'$ denoted by the set of viscosity solutions $\om$ of \eqref{eq:hj-e-cri} satisfying 
\beq
\int_{T\T^n}\partial_{u}L(x,v,0)\om d\mu\geq 0,\quad\forall \,\mu\in\mathcal M'.
\eeq
Here 
\ben
L : T\T^n\times\R &\longrightarrow& \R \\
(x,v,u)&\longrightarrow& \max_{p\in T_x^*\T^n}\big\{\langle v,p\rangle -H(x,p,u)\big\}.
\een
 is the {\it  Lagrangian} associated with $H$ and $\mathcal M'$ is a selected set of all {\it stochastic Mather measures} associated with \eqref{eq:hj-e-cri} (see Sec. \ref{s2} for the definition of $\cM'$).
\end{thm}

\subsection{Background and strategy}

The establishment of $c(H)$ in \eqref{eq:hj-e-cri} was first studied by Lions, Papanicolaou and Varadhan \cite{LPV} for the case $\alpha(x) \equiv 0$ by using a homogenization approach. Precisely, they considered the following {\it discounted Hamilton-Jacobi equation}
\beq\label{eq:hj-dis}
\lb u+H(x,d_xu)=0, \quad x\in\T^n
\eeq
of which the viscosity solution $u_\eps$ is unique for any $\lb>0$, and showed the convergence of $\lb u_\lb$ to $-c(H)$ as $\lb\rightarrow 0_+$. That naturally leads to a question whether $u_\lb$ converges as $\lb\rightarrow 0_+$ as well. Such a question was firstly addressed by Gomes \cite{G2}, Iturriaga and Sanchez-Morgado \cite{ISM} under certain restricted assumptions. Afterwards, Davini, Fathi, Iturriaga and Zavidovique \cite{DFIZ}
gave a confirmed answer to this question in the case of \eqref{eq:hj-dis}. Their approach relies on a dynamical characterization of the Mather measures in light of the weak KAM theory. Following the same approach, other convergence problems were gradually considered in \cite{AAIY,CCIZ,IS,Tu}. Recently, the convergence of viscosity solution was proved in \cite{WYZ} for \eqref{eq:hj-e} with $\alpha(x)\equiv 0$ (as $\lb\rightarrow 0_+$), under Tonelli assumptions on $H(x,p,u)$. This work firstly proposed a method to characterize the Mather measures for first order Hamilton-Jacobi equations nonlinear of $u$ (also called {\it contact Hamilton-Jacobi equations} in the context of \cite{L,WWY}). Very recently works toward this topic can be found in \cite{C,Z}.

At the same time, the convergence of viscosity solutions  for the discounted Hamilton-Jacobi equations with degenerate diffusion was achieved by Mitake and Tran \cite{MT}. They present a novel characterization of the stochastic Mather measures by using the {\it adjoint method} developed in \cite{E}, which is different from the former definition of Mather measures given in \cite{G,ISM2}. Later, Ishii, Mitake and Tran also gave a general criterion to deal with similar convergence problems (as the discounted limit) in \cite{IMT1,IMT2}. These two works successfully used a linear programming method to define the Mather measures for equations of the form 
\[
H(x,d_x u,D^2u)+\lb u=0, \quad x\in\T^n,\lb>0.
\] 
However, this concise method can not apply to the case with Hamiltonians nonlinear of $u$ directly. That urges us to find other ways to characterize the Mather measures for \eqref{eq:hj-e}, then prove the convergence of viscosity solution $u_\lb$ as $\lb\rightarrow 0_+$.

In this paper, we first prove the uniform boundedness and equi-Lipschitzness of $\{u_\lb\}_{\lb\in(0,1]}$ in the spirit of {\it Bernstein's method}, see \cite{AT,Bar,Ber}. That gives us chance to get the convergence of $\{u_\lb\}_{\lb\in(0,1]}$ along subsequences, in view of the {\it Arzela-Ascoli Theorem}. On the other side, we can define the stochastic Mather measure associated with \eqref{eq:hj-e} by using the adjoint method. Besides, \cite{WYZ} also supplies a methodology to verify the asymptotic properties of the stochastic Mather measures as $\lb\rightarrow 0_+$ for Hamiltonians nonlinear of $u$ (see Proposition. \ref{prop:conv-mat-meas}). Consequently, that supplies a criterion to describe the accumulating points of $\{u_\lb\}_{\lb\in(0,1]}$ as $\lb\rightarrow 0_+$ (see Lemma \ref{lem:up} and Lemma \ref{lem:low}), then further indicates the accumulating point is unique.

\subsection{Organization of the article.}\label{s1.6}

 In Section \ref{s2}, we introduce the concept of stochastic Mather measures and a viewpoint of adjoint equations dealing with it. In Section \ref{s3}, we get a qualitative estimate of the solutions of \eqref{eq:hj-e}, then finally prove Theorem \ref{thm:1}. As a necessary complement, we give the Remark \ref{rmk:main} to elaborate the significance of 
 Theorem \ref{thm:1}. For the readability and consistency of this article, some lengthy independent conclusions are moved to Appendix.

\vspace{10pt}

\noindent{\bf Acknowledgement.} This work is supported by the National Natural Science Foundation of China (Grant No. 11901560). The author would like to thank the Laboratory of Mathematics for Nonlinear Science, Fudan University (LNMS) for the hospitality, where this research was initiated during the author's visiting in April 2021. 

\vspace{20pt}

\section{Stochastic Mather measures of \eqref{eq:hj-e-cri}}\label{s2}

Due to the assumptions (H1)-(H3), any viscosity solution $\om$ of \eqref{eq:hj-e-cri} has to be continuous, so we can shift it to the following two functions:
\[
\check{\om}(x):=\om(x)+|\om(x)|_{L^\infty}\geq 0,\quad \hat\om(x):=\om(x)-|\om(x)|_{L^\infty}\leq 0.
\]
Consquently, for any $\lb\in(0,1]$, we can verify that $\check\om$ (resp. $\hat\om$) is a supersolution (resp. subsolution) of \eqref{eq:hj-e}. By applying the {\it Perron's method} (see \cite{CIL} for instance), the following defined function 
\beq
u_\lb(x):=\sup\{v(x): \hat\om(x)\leq v(x)\leq \check\om(x), \text{$v(x)$ is a subsolution of \eqref{eq:hj-e}}\}
\eeq
is the unique viscosity solution of \eqref{eq:hj-e}. Therefore, there exists a constant $C_{p}>0$ such that
\be\label{eq:uni-bound}
|u_\lb(x)|\leq C_p,\quad\forall\,\lb\in(0,1].
\ee
Moreover, we can also prove that $\{u_\lb\}_{\lb\in(0,1]}$ are uniformly Lipschitz:
\begin{prop}
[Bernstein's method] 
\label{prop:bern}
Under the assumptions {\rm (H1)-(H5)}, there exists a constant $C_{\rm Lip}>0$ such that for any $\lb\in(0,1]$, the viscosity solution $u_\lb$ of \eqref{eq:hj-e} satisfies 
\beq\label{eq:uni-lip}
|d_x u_\lb|_{L^\infty}\leq C_{\rm Lip}.
\eeq
\end{prop}
\begin{proof}
The proof relies on the idea of  Bernstein's method developped in \cite{Ber,Bar} but need more arguments due to the low regularity of $u_\lb$. Roughly speaking, we will show that a power of $|d_xu_\lb|_{L^\infty}$ should be a subsolution  (in a weak sense) of certain elliptic equation, so the comparison principle constrains $|d_xu_\lb|_{L^\infty}$ from above. Such a procedure has been proved to be successful in \cite{DLP,AT} for Hamiltonians linear of $u$. Now we adapt their ideas to more generalized Hamiltonians.

Without loss of generality, we can endow $\T^n$ with a coordinate, 
so it surffices to prove that $|d_xu_\lb|_{L^\infty}$ is uniformly bounded in the domain $B_1:=\{x\in\T^n: |x|\leq 1\}$ for $\lb\in(0,1]$. In other words, we just need to show that there exists a uniform constant $L>0$, such that for any $\lb\in(0,1]$ and $\hat x\in {\rm int} B_1$ (the interior of $B_1$), 
\[
\limsup_{x\rightarrow\hat x}\frac{u_\lb(\hat x)-u_\lb(x)}{|\hat x-x|}\leq L.
\]
Otherwise, for any $L\geq 1$ sufficiently large, we could always find $\lb\in(0,1]$ and $\hat x\in {\rm int} B_1$ such that 
\be\label{eq:assu-0}
\limsup_{x\rightarrow\hat x}\frac{u_\lb(\hat x)-u_\lb(x)}{|\hat x-x|}> L,
\ee
we will show that leads to a contradiction by obtaining an upper bound for $L$.

\underline{Step 1}. If there exist a $L\geq 1$ and $x_0,y_0\in {\rm int} B_1$  such that 
\be\label{eq:assu-1}
u_\lb(x_0)-u_\lb(y_0)-L|x_0-y_0|=\sup_{x,y\in B_1}\Big(u_\lb(x)-u_\lb(y)-L|x-y|\Big)>0,
\ee
then $x_0\neq y_0$. Due to Lemma 3.2 of \cite{CIL}, for any $\eps>0$, there exist $X_\eps,Y_\eps\in\mathbb S(n)$ (the set of $n\times n-$symmetric matrices) satisfying 
\be\label{eq:toy-matrix} 
\begin{pmatrix}
   X_\eps   & 0   \\
    0  &  -Y_\eps
\end{pmatrix}\leq J+\eps J^2
\ee
where `$\leq$' is the usual order of $\mathbb S(n)$ and 
\[
J_{2n\times 2n}:=\frac{L}{|x_0-y_0|}\begin{pmatrix}
   Z   &   -Z \\
    -Z  &  Z
\end{pmatrix},\quad Z:=I_{n\times n}-\frac{x_0-y_0}{|x_0-y_0|}\otimes\frac{x_0-y_0}{|x_0-y_0|}
\]
such that 
\be\label{eq:toy-ineq}
& &H(x_0,L\frac{x_0-y_0}{|x_0-y_0|},\lb u_{\lb}(x_0))-\alpha(x_0){\rm tr} (X_\eps)\\
&\leq& 0\leq H(y_0,L\frac{x_0-y_0}{|x_0-y_0|},\lb u_{\lb}(y_0))-\alpha(y_0){\rm tr} (Y_\eps).\nonumber
\ee
On the other side, for any $s>1$, we define a nonnegative matrix 
\[
A_s:=\frac12 \begin{pmatrix}
   s^2 \alpha(x_0) I_{n\times n}  &   s\sqrt{\alpha(x_0)\alpha(y_0)} I_{n\times n}\\
    s\sqrt{\alpha(y_0)\alpha(x_0)} I_{n\times n}  &  \alpha(y_0) I_{n\times n} 
\end{pmatrix}
\]
and multiply it to \eqref{eq:toy-matrix} on the right, then we get 
\[
{\rm tr}(s^2\alpha(x_0)X_\eps-\alpha(y_0)Y_\eps)\leq{\rm tr}(JA_s)+\eps{\rm tr}(J^2A_s).
\]
Combining this inequality with \eqref{eq:toy-ineq} then making $\eps\rightarrow 0_+$, we get 
\be\label{eq:toy-key}
s^2 H(x_0,L\frac{x_0-y_0}{|x_0-y_0|},\lb u_{\lb}(x_0))-H(y_0,L\frac{x_0-y_0}{|x_0-y_0|},\lb u_{\lb}(y_0))\leq {\rm tr}(JA_s).
\ee
On one side, by computation we get 
\[
{\rm tr}(JA_s)\leq \frac{L \big|\sqrt{\alpha(x)}\big|_{C^1}^2}{|x_0-y_0|}\Big((s-1)^2+|x_0-y_0|^2\Big)
\]
On the other side, if we set $s^2=1+\beta|x_0-y_0|$ with $\beta>0$, 
\ben
& &s^2 H(x_0,L\frac{x_0-y_0}{|x_0-y_0|},\lb u_{\lb}(x_0))-H(y_0,L\frac{x_0-y_0}{|x_0-y_0|},\lb u_{\lb}(y_0))\\
&=&(s^2-1) H(x_0,L\frac{x_0-y_0}{|x_0-y_0|},\lb u_{\lb}(x_0))\\
& &+ H(x_0,L\frac{x_0-y_0}{|x_0-y_0|},\lb u_{\lb}(x_0))-H(x_0,L\frac{x_0-y_0}{|x_0-y_0|},\lb u_{\lb}(y_0))+\\
& &H(x_0,L\frac{x_0-y_0}{|x_0-y_0|},\lb u_{\lb}(y_0))-H(y_0,L\frac{x_0-y_0}{|x_0-y_0|},\lb u_{\lb}(y_0))\\
&\geq&(s^2-1) H(x_0,L\frac{x_0-y_0}{|x_0-y_0|},\lb u_{\lb}(x_0))-\\
& &\big|H(x_0,L\frac{x_0-y_0}{|x_0-y_0|},\lb u_{\lb}(y_0))-H(y_0,L\frac{x_0-y_0}{|x_0-y_0|},\lb u_{\lb}(y_0))\big|\\
&\geq&(s^2-1) \Big(H(x_0,L\frac{x_0-y_0}{|x_0-y_0|},0)-C_p\rho^*\Big)-\kappa(C_p)\Big(H(x_0,L\frac{x_0-y_0}{|x_0-y_0|},0)+\varsigma(C_p)\Big)|x_0-y_0|\\
&=&\Big((\beta-\kappa(C_p)) H(x_0,L\frac{x_0-y_0}{|x_0-y_0|},0))-\beta C_p\rho^*-\kappa(C_p)\varsigma(C_p)\Big)|x_0-y_0|
\een
in which the first inequality is due to (H3) and the second inequality is due to (H4). Turning back to \eqref{eq:toy-key}, we finally get 
\[
(\beta-\kappa(C_p)) H(x_0,L\frac{x_0-y_0}{|x_0-y_0|},0))-\beta C_p\rho^*-\kappa(C_p)\varsigma(C_p)\leq L(1+\beta^2)\big|\sqrt{\alpha(x)}\big|_{C^1}^2.
\]
As long as $\beta>\kappa(C_p)$, $L$ has to be upper bounded due to (H2) (by some constant depending only on $C_p,\rho^*$).

\underline{Step 2}. If the supremum of \eqref{eq:assu-1} can not be obtained, then we have to make use of certain cutoff function and slightly modify \eqref{eq:assu-1} to make the supremum available. Precisely, we pick a positive smooth function $\phi: {\rm int}B_{3/2}\subsetneq\T^n\rightarrow[1,+\infty)$ which satisfies $\phi\equiv 1$ on $B_1$, $\lim_{x\rightarrow\partial B_{3/2}}\phi(x)\rightarrow+\infty$ and 
\[
 \forall\, x\in {\rm int}B_{3/2},\;\left\{
\begin{aligned}
& |d_x\phi(x)|\leq C(\phi(x))^{m},\\
& |D^2\phi(x)|\leq C(\phi(x))^{2m-1}
\end{aligned}
\right.
\]
for some constant $C>0$. Without loss of generality, we can assume (H3) holds for $1<m\leq2$. Actually, any regularization of the map $x\rightarrow \max\{(2{\rm dist}(x,\partial B_{3/2}))^{-\frac1{m-1}},1\}$ can be such a cutoff function. Benefiting from it, for any $a>0$ sufficiently small, there always exist $x_a,y_a\in{\rm int} B_{3/2}$ such that 
\be\label{eq:assu-2}
& &u_\lb(x_a)-u_\lb(y_a)-L\phi(y_a)|x_a-y_a|-\frac1{2a}|x_a-y_a|^2\\
&=&\sup_{x,y\in B_{3/2}}\Big(u_\lb(x)-u_\lb(y)-L\phi(y)|x-y|-\frac1{2a}|x-y|^2\Big)>0.\nonumber
\ee
This conclusion was firstly proved in Theorem 3.1 of \cite{AT}, but for the consistency we sketch their procedure here: \eqref{eq:assu-0} implies for any $a>0$, we can find $x,y\in B_1$ such that 
\[
u_\lb(x)-u_\lb(y)-L\phi(y)|x-y|-\frac1{2a}|x-y|^2>0,
\]
so the supreme has to be positive. Notice that $y\notin\partial B_{3/2}$ and
\[
\sup_{x,y\in B_{3/2}}\Big(u_\lb(x)-u_\lb(y)-L\phi(y)|x-y|-\frac1{2a}|x-y|^2\Big)\leq \sup_{x,y\in B_{3/2}}|u_\lb(x)-u_\lb(y)|,
\]
so $|x_a-y_a|\geq d_{\lb}(a)>0$ for some constant $d_\lb(a)$ due to the uniform continuity of $u_\lb$ on $ B_{3/2}$. Furthermore,
\[
0<\sup_{x,y\in B_{3/2}}\Big(u_\lb(x)-u_\lb(y)-L\phi(y)|x-y|-\frac1{2a}|x-y|^2\Big)\leq \sup_{y\in B_{3/2}}\Big({\rm osc}_{B_{3/2}}u_\lb-L\phi(y)d_\lb(a)\Big)
\]
imposes ${\rm dist}(y_a,\partial B_{3/2})\geq \frac{L d_\lb(a)}{2C_p}$ since ${\rm osc}_{B_{3/2}}u_\lb\leq 2C_p<+\infty$. If $x_a\in\partial B_{3/2}$, then 
\ben
u_\lb(x_a)-u_\lb(y_a)-L\phi(y_a)|x_a-y_a|-\frac1{2a}|x_a-y_a|^2&\leq& 2C_p-L\frac{C|x_a-y_a|}{{\rm dist}(y,\partial B_{3/2})}\\
&\leq& 2C_p-L C<0
\een
as long as $L>2C_p/C$. That contradicts the positiveness of the supreme, so $x_a,y_a\in {\rm int} B_{3/2}$ is proved. Moreover, we get
\[
|x_a-y_a|\leq 2\sqrt{C_p a},\quad\forall \,a>0
\]
and
\be\label{eq:bypro}
\limsup_{a\rightarrow 0_+}(L\phi(y_a)|x_a-y_a|+\frac1{2a}|x_a-y_a|^2)\leq \limsup_{a\rightarrow 0_+}\sup_{\substack{x,y\in B_{3/2}\\|x-y|\leq2\sqrt{C_pa}}}\{u(y)-u(x)\}=0
\ee
as a byproduct for later use. Benefiting from \eqref{eq:assu-2}, for any $\eps>0$ and sufficiently small $a>0$, once again we take $X_{\eps,a},Y_{\eps,a}\in\mathbb S(n)$ satisfying 
\be\label{eq:toy-matrix-2} 
\begin{pmatrix}
   X_\eps   & 0   \\
    0  &  -Y_\eps
\end{pmatrix}\leq J_a+\eps J_a^2
\ee
with
\[
J_{a}:=\underbrace{\frac{L\phi(y_a)}{|x_a-y_a|}\begin{pmatrix}
   Z_1   &   -Z_1 \\
    -Z_1  &  Z_1
\end{pmatrix}+\frac1a\begin{pmatrix}
    I_{n\times n}  &   - I_{n\times n} \\
    - I_{n\times n}  &   I_{n\times n}
\end{pmatrix}}_{:=J'_a}+\underbrace{L\begin{pmatrix}
  0    &  Z_2  \\
    Z_2^t  &  Z_3
\end{pmatrix}}_{:=J''_a}
\]
and
\[\left\{
\begin{aligned}
&Z_1:=I_{n\times n}-\frac{x_a-y_a}{|x_a-y_a|}\otimes\frac{x_a-y_a}{|x_a-y_a|},\\
&Z_2:=d_x\phi(y_a)\otimes\frac{x_a-y_a}{|x_a-y_a|},\\
&Z_3:=-(Z_2+Z_2^t)+D^2\phi(y_a)|x_a-y_a|.
\end{aligned}
\right.
\]
such that 
\be\label{eq:toy-ineq-2}
& &H\Big(x_a,\big(L\phi(y_a)+\frac{|x_a-y_a|}{a}\big)\frac{x_a-y_a}{|x_a-y_a|},\lb u_{\lb}(x_a)\Big)-\alpha(x_a){\rm tr} (X_{\eps,a})\leq 0\\
&\leq& H\Big(y_a,\underbrace{\big(L\phi(y_a)+\frac{|x_a-y_a|}{a}\big)\frac{x_a-y_a}{|x_a-y_a|}}_{:=P_a}-\underbrace{L|x_a-y_a|d_x\phi(y_a)}_{:=Q_a}, \lb u_{\lb}(y_a)\Big)-\alpha(y_a){\rm tr} (Y_{\eps,a}).\nonumber
\ee
Similarly, for any $s>1$, we define a nonnegative matrix 
\[
A_s:=\frac12 \begin{pmatrix}
   s^2 \alpha(x_a) I_{n\times n}  &   s\sqrt{\alpha(x_a)\alpha(y_a)} I_{n\times n}\\
    s\sqrt{\alpha(y_a)\alpha(x_a)} I_{n\times n}  &  \alpha(y_a) I_{n\times n} 
\end{pmatrix}
\]
and multiply it to \eqref{eq:toy-matrix-2} on the right, then we get 
\[
{\rm tr}(s\alpha(x_a)X_{\eps,a}-\alpha(y_a)Y_{\eps,a})\leq{\rm tr}(J_aA_s)+\eps{\rm tr}(J_a^2A_s).
\]
Combining this inequality with \eqref{eq:toy-ineq-2} then making $\eps\rightarrow 0_+$, we get 
\be\label{eq:toy-key-2}
s H(x_a,P_a,\lb u_{\lb}(x_a))-H(y_a,P_a-Q_a,\lb u_{\lb}(y_a))\leq {\rm tr}(J_aA_s).
\ee
The right hand side can be estimated by 
\ben
{\rm tr}(J_aA_s)&=&{\rm tr}(J'_aA_s)+{\rm tr}(J''_aA_s)\\
&\leq&\big|\sqrt{\alpha(x)}\big|_{C^1}^2|P_a|(1+\beta^2)|x_a-y_a|+\big|\sqrt{\alpha(x)}\big|_{C^1}^2(1+\beta)|Q_a| \\
& &+\frac12 L \big|\sqrt{\alpha(x)}\big|_{C^1}^2|D^2\phi(y_a)|\cdot|x_a-y_a|
\een
in view of $|P_a|=L\phi(y_a)+|x_a-y_a|/a$, $|Q_a|=L|x_a-y_a|\cdot|d_x\phi(y_a)|$ and $s:=1+\beta|x_a-y_a|$ ($\beta>0$).  On the other side, the left hand side of \eqref{eq:toy-key-2} satisfies
\ben 
& &s H(x_a,P_a,\lb u_{\lb}(x_a))-H(y_a,P_a-Q_a,\lb u_{\lb}(y_a))\\
&=&(s-1) H(x_a,P_a,\lb u_{\lb}(x_a))+ H(x_a,P_a,\lb u_{\lb}(x_a))- H(x_a,P_a,\lb u_{\lb}(y_a))+\\
& & H(x_a,P_a,\lb u_{\lb}(y_a))- H(y_a,P_a,\lb u_{\lb}(y_a))+H(y_a,P_a,\lb u_{\lb}(y_a))-H(y_a,P_a-Q_a,\lb u_{\lb}(y_a))\\
&\geq&(s-1) H(x_a,P_a,\lb u_{\lb}(x_a))-\kappa(C_p)\Big(H(x_a,P_a,0)+\varsigma(C_p)\Big)|x_a-y_a|
\een
\ben
& &-|H(y_a,P_a,\lb u_{\lb}(y_a))-H(y_a,P_a-Q_a,\lb u_{\lb}(y_a))|\\
&\geq&(s-1) \Big(H(x_a,P_a,0)-C_p\rho^*\Big)-\kappa(C_p)\Big(H(x_a,P_a,0)+\varsigma(C_p)\Big)|x_a-y_a|\\
& &-\xi(C_p)\Big(H(y_a, P_a, \lb u_\lb(y_a))+\eta(C_p)\Big)\frac{|Q_a|}{|P_a|+1}\\
&\geq&\Big(\beta-\kappa(C_p) -\xi(C_p)[1+\kappa(0)]C\phi^{m-1}(y_a)\Big) H(x_a,P_a,0))|x_a-y_a|\\
& &-\Big(\beta C_p\rho^*+\kappa(C_p)\varsigma(C_p)+\xi(C_p)[\kappa(0)\varsigma(0)+C_p\rho^*+\eta(C_p)]C\phi^{m-1}(y_a)\Big)|x_a-y_a|
\een
in which the first inequality is due to a similar argument as in Step 1, the second inequality is due to (H5) and the lst inequality is due to 
\[
|Q_a|\leq C\phi^{m-1}(y_a)|x_a-y_a|\cdot |P_a|\leq CL^{1-m}|P_a|^m|x_a-y_a|.
\]
Due to \eqref{eq:bypro} and (H2), there holds
\ben
& &K_m\Big[\Big(\beta-\kappa(C_p) -\xi(C_p)[1+\kappa(0)]C\phi^{m-1}(y_a)\Big)\\
& &-\big|\sqrt{\alpha(x)}\big|_{C^1}^2(1+\beta)CL^{1-m}-\frac12\big|\sqrt{\alpha(x)}\big|_{C^1}^2CL^{1-m}\phi^{m-1}(y_a)\Big]|P_a|^m\\
&\leq&M_m\Big(\beta-\kappa(C_p) -\xi(C_p)[1+\kappa(0)]C\phi^{m-1}(y_a)\Big)+(1+\beta^2)\big|\sqrt{\alpha(x)}\big|_{C^1}^2|P_a|.
\een
Without loss of generality, we can assume $L\geq (2C\big|\sqrt{\alpha(x)}\big|_{C^1}^2)^{\frac1{(m-1)}}$. Consequently, we take 
\[
\beta=2\kappa(C_p)+1+4\Big(\xi(C_p)(1+\kappa(0))C+\frac14\Big)\phi^{m-1}(y_a),
\]
then further get 
\ben
& &K_m\Big(\xi(C_p)(1+\kappa(0))C+\frac14\Big)\phi^{m-1}(y_a)|P_a|^m\\
&\leq& M_m\Big(\kappa(C_p)+1+[3\xi(C_p)[1+\kappa(0)]C+1]\phi^{m-1}(y_a)\Big)+(1+\beta^2)\big|\sqrt{\alpha(x)}\big|_{C^1}^2|P_a|.
\een
Dividing both sides by $\Big(\xi(C_p)(1+\kappa(0))C+\frac14\Big)\phi^{m-1}(y_a)$, we get 
\ben
K_m|P_a|^m&\leq& M_m\Big(4\kappa(C_p)+12C\xi(C_p)(1+\kappa(0))+8\Big)\\
& &+\Big(4\big|\sqrt{\alpha(x)}\big|_{C^1}^2[1+(1+2\kappa(C_p))^2]+8(1+2\kappa(C_p))\Big)|P_a|\\
& &+4(1+4\xi(C_p)(1+\kappa(0))C)L^{1-m}|P_a|^m.
\een
By strengthening the second restriction to $L\geq \Big(\dfrac{8(1+4\xi(C_p)(1+\kappa(0))C)}{K_m}\Big)^{\frac1{m-1}}$, we obtain
\ben
\frac{K_m}2|P_a|^m&\leq&  M_m\Big(4\kappa(C_p)+12C\xi(C_p)(1+\kappa(0))+8\Big)\\
& &+\Big(4\big|\sqrt{\alpha(x)}\big|_{C^1}^2[1+(1+2\kappa(C_p))^2]+8(1+2\kappa(C_p))\Big)|P_a|
\een
which imposes $L\leq |P_a|\leq C_*<+\infty$ for some constant $C_*=C_*(C_p,K_m,M_m,m,\big|\sqrt{\alpha(x)}\big|_{C^1})$.
\end{proof}

Above preliminaries guarantee the convergence of $u_\lb$ along subsequences, in view of the Arzela-Ascoli Theorem. To show whether this convergence holds for the whole sequence $\lb\rightarrow 0_+$ or not, the following definition is needed.

\begin{defn}[Mather measure]
Denote by $\cP(T\T^n)$ the set of probability measures on $T\T^n$. A probability measure $\mu\in\cP(T\T^n)$ is called a {\it stochastic Mather measure}, if it satisfies:
\begin{itemize}
\item $\int_{TM} L(x,v,0)d\mu(x,v)=c(H)$;
\item $\int_{TM}\langle v,\nabla\varphi(x)\rangle-\alpha(x)\Dt\varphi(x) d\mu(x,v)=0$, for any $\varphi(x)\in C^2(\T^n,\R)$.
\end{itemize}
We denote by $\mathcal M$ the set of all stochastic Mather measures. Next, we will show how to get the stochastic Mather measures and use them to describe the variational properties of $u_\lb$.
\end{defn}

\subsection{Adjoint equation of \eqref{eq:hj-e}}

Evans firstly introduced the nonlinear adjoint method for first order Hamilton-Jacobi equations to study the vanishing viscosity process. Afterwards, in the works \cite{MT,T} this method was used to give significant estimate about the viscosity solutions, even for nonconvex Hamiltonians. Following their procedure, for each $\eta>0$, we consider the approximation of \eqref{eq:hj-e} as
\beq\label{eq:hj-e-app}\tag{HJ$_{e}^{\eta}$}
H(x,d_xu_\lb^\eta,\lb u_\lb^\eta)=(\alpha(x)+\eta^2)\Dt u_\lb^\eta,\quad  x\in \T^n.
\eeq
By a standard analysis, the following estimate can be proved:
\begin{lem}\label{lem:visc-est}
Let $u_\lb^\eta$ and $u_\lb$ be the solutions of \eqref{eq:hj-e-app} and \eqref{eq:hj-e} respectively. There exists a constant $C' > 0$ independent of $\lb,\eta\in (0,1]$ such that
\beq\label{eq:visc-est}
|u_\lb^\eta-u_\lb|_{L^\infty}\leq C'\frac\eta\lb.
\eeq
\end{lem}
Due to this Lemma, we can introduce the associated adjoint equation of the linearized operator of \eqref{eq:hj-e} by:
\beq\tag{AJ$_{e}$}\label{eq:hj-ad}
\lb \partial_u H(x,d_x u_\lb^\eta,0)\theta_\lb^\eta-{\rm div}\big(\partial_p H(x,d_xu_\lb^\eta,0)\theta_\lb^\eta\big) =\Dt(\alpha(x)\theta_\lb^\eta)+\eta^2\Dt\theta_\lb^\eta+\lb \delta_{x_0}
\eeq
for some $x_0\in \T^n$ and $\dt_{x_0}\in\cP(\T^n)$ being the Dirac measure at this point. We can also prove that
\beq\label{eq:meas-est} 
\theta_\lb^\eta\geq 0, \quad \int_{\T^n} \partial_u H(x,d_xu_\lb^\eta,0)\theta_
\lb^\eta(x) dx=1.
\eeq
For the readability, we postpone the proof of Lemma \eqref{lem:visc-est} and \eqref{eq:meas-est} to Appendix \ref{a1}, and use them without any doubt in this section.

For any $\lb,\eta>0$, we get a probability measure $\nu_\lb^\eta\in\cP(T^*\T^n)$ via
\beq
\frac{\int_{\T^n} f(x, d_xu_\lb^\eta(x))\theta_\lb^\eta(x)dx}{\int_{\T^n} \theta_\lb^\eta(x)dx}=\iint_{T^*\T^n}f(x,p)d\nu_\lb^\eta(x,p)
\eeq
for all $f\in C_c(T^*\T^n,\R)$. We can pull back $\nu_\lb^\eta$ to a probability measure $\mu_\lb^\eta\in\cP(T\T^n)$ with respect to  the {\it Legrendre tranformation}
\[
\cL:(x,v)\in T\T^n\longrightarrow (x,\partial_vL(x,v,0))\in T^*\T^n, 
\]
i.e. $\mu_\lb^\eta:=\cL^*\nu_\lb^\eta$ satisfies 
\be\label{eq:mea-tran}
\iint_{T^*\T^n}f(x,p)d\nu_\lb^\eta(x,p)=\int_{T\T^n} f(x,\partial_vL(x,v,0)) d\mu_\lb^\eta(x,v)
\ee
for all $f\in C_c(T^*\T^n,\R)$. 

\begin{prop}\label{prop:conv-mat-meas}
Any weak* limit $\mu$ of $\mu_\lb^\eta$ as $\lb,\eta\rightarrow 0_+$ has to be a stochastic Mather measure. 
\end{prop}
\begin{proof} By a simple deduction, we get 
\ben
& &\langle\partial_pH(x,d_x u_\lb^\eta,0), d_x u_\lb^\eta\rangle-H(x,d_x u_\lb^\eta,0)\\
& &+\lb u_\lb^\eta\int_0^1[\partial_u H(x,d_x u_\lb^\eta,0)-\partial_u H(x,d_x u_\lb^\eta,s\lb u_\lb^\eta)]ds\\
&=&\langle\partial_pH(x,d_x u_\lb^\eta,0), d_x u_\lb^\eta\rangle+\partial_u H(x,d_x u_\lb^\eta,0)\lb u_\lb^\eta-(\alpha(x)+\eta^2)\Dt u_\lb^\eta.
\een
Multiplying both sides with $\theta_\lb^\eta$ and integrating them over $\T^n$, it yields 
\ben
& &\lb u_\lb^\eta(x_0)-\int_{\T^n}\Big(\langle\partial_pH(x,d_x u_\lb^\eta,0), d_x u_\lb^\eta\rangle-H(x,d_x u_\lb^\eta,0)\Big)\theta_\lb^\eta dx\\
&=&\int_{\T^n} \lb u_\lb^\eta\theta_\lb^\eta\Big(\int_0^1[\partial_u H(x,d_x u_\lb^\eta,0)-\partial_u H(x,d_x u_\lb^\eta,s\lb u_\lb^\eta)]ds\Big) dx
\een
which further implies
\ben
& &\bigg|\lb u_\lb^\eta(x_0)-\int_{\T^n}\Big(\langle\partial_pH(x,d_x u_\lb^\eta,0), d_x u_\lb^\eta\rangle-H(x,d_x u_\lb^\eta,0)\Big)\theta_\lb^\eta dx\bigg| \\
&=&\bigg|\lb u_\lb^\eta(x_0)-\int_{\T^n}\theta_\lb^\eta(x)dx \int_{T^*\T^n}\langle\partial_pH(x,p,0), p\rangle-H(x,p,0) d\nu_\lb^\eta(x,p)\bigg|\\
&=&\bigg|\lb u_\lb^\eta(x_0)-\int_{\T^n}\theta_\lb^\eta(x)dx \int_{T\T^n}L(x,v,0) d\mu_\lb^\eta(x,v)\bigg|\\
&\leq&\int_{\T^n} |\lb u_\lb^\eta|\int_0^1\bigg|\partial_u H(x,d_x u_\lb^\eta,0)-\partial_u H(x,d_x u_\lb^\eta,s\lb u_\lb^\eta)\bigg|ds\cdot\theta_\lb^\eta dx\\
&\leq& \int_{\T^n}( C'\eta+\lb C_p)\cdot  2\rho^* \theta_\lb^\eta(x) dx
\een
due to \eqref{eq:visc-est}. 
Taking  $\eta,\lb\rightarrow 0_+$ we derive that 
\[
\int_{T\T^n}L(x,v,0) d\mu(x,v)=0.
\]
On the other side, if we multiply \eqref{eq:hj-ad} by any given $\varphi\in C^2(\T^n,\R)$ then integrate over $\T^n$, we get 
\ben
& &\int_{\T^n}\Big(\langle\partial_pH(x,d_x u_\lb^\eta,0), d_x \varphi\rangle-(\alpha(x)+\eta^2)\Dt\varphi\Big)\theta_\lb^\eta dx\\
&=&
\lb \varphi(x_0)-\lb\int_{\T^n}\partial_uH(x,d_xu_\lb^\eta,0)\varphi\theta_\lb^\eta dx.
\een
Similarly as above that indicates
\[
\int_{T\T^n}\langle v,\nabla\varphi(x)\rangle-\alpha(x)\Dt\varphi(x) d\mu(x,v)=0
\]
as  $\eta,\lb\rightarrow 0_+$.
\end{proof}
 \begin{defn}
From now on, we denote by $\mathcal M'$ the set of all weak* limit of $\mu_\lb^\eta$ defined by \eqref{eq:mea-tran}, then this Proposition implies $\mathcal M'\subset\mathcal M$.\end{defn}
 
 \medskip
 
\section{Qualitative exploration of the viscosity solution of \eqref{eq:hj-e}}\label{s3}

\begin{lem}\label{lem:com}
For $\lb\in(0,1]$, there exists a mollifier $\zeta\in C_c^\infty(\R^n,\R)$ satisfying 
\[
\zeta\geq 0,\;\; {\rm supp}(\zeta)\subset \ol{B(0,1)},\quad|\zeta|_{L^1(\R^n,\R)}=1,
\]
such that for any suitably small $\eta>0$, the function 
\beq\label{eq:cri-sol-mol}
\om_\lb^\eta(x):=\int_{\R^n}\underbrace{\eta^{-n}\zeta(\eta^{-1}y)}_{:=\zeta^\eta(y)}u_\lb(x+y)dy,\quad x\in \T^n
\eeq
satisfies
\[
H(x,d_x\om_\lb^\eta(x),\lb u_\lb(x))\leq \alpha(x)\Dt\om_\lb^\eta(x)+S^\eta(x),\quad x\in \T^n
\]
for some continuous function $S^\eta:\T^n\rightarrow\R$. Moreover, there exists a constant $C>0$ such that 
\beq
|S^\eta(x)|\leq C,\;\;|S^\eta|_{L^\infty}\leq C\sqrt\eta,\;\;|\eta^2\Dt\om^\eta|\leq C\eta.
\eeq
\end{lem}
\begin{proof}
We follow the procedure in \cite{MT} but with necessary adaptions. Firstly, due to Proposition \ref{prop:bern}, there exists a constant $C_1>0$ uniform for $\lb\in(0,1]$ such that 
\[
-C_1<\alpha(x)\Dt u_\lb(x)\leq C_1,\quad\text{for any $x\in \T^n$ in the sense of viscosity}.
\]
Then due to \cite{I}, we get 
\beq\label{eq:upper-ishii}
|d_xu_\lb(x)|_{L^\infty}+|\alpha(x)\Dt u_\lb(x)|_{L^\infty}\leq C_2
\eeq
for some constant $C_2>0$. Secondly, we show that  $u_\lb(x)$ is a subsolution of \eqref{eq:hj-e} in the distributional sense, due to the ideas in \cite{J,JLS}. Precisely, let $\ol\om_\lb^\dt:=\sup_{y\in\R^n}\big(u_\lb(y)-\frac{|x-y|^2}{2\dt}\big)$ being the sup-convolution of $u_\lb$ for each $\dt>0$, then $\ol\om_\lb^\dt$ should be semi-convex and a viscosity subsolution of the following
\beq\label{eq:sup-conv}
H(x,d_x\ol\om_\lb^\dt,\lb u_\lb(x))\leq \alpha(x)\Dt\ol\om_\lb^\dt(x)+\varpi(\dt),\quad x\in\T^n
\eeq
for some modulus of continuity $\varpi:(0,+\infty)\rightarrow\R$ satisfying $\lim_{\dt\rightarrow 0_+}\varpi(\dt)=0$. Since $\ol\om_\lb^\dt$ is a semi-convex function, it is twice differentiable almost everywhere of $\T^n$. In view of \eqref{eq:upper-ishii}, $\ol\om_\lb^\dt$ is a distributional subsolution of \eqref{eq:sup-conv}, then passing to a subsequence if necessary, there hold
\[
\ol\om_\lb^\dt\rightarrow u_\lb,\quad\text{uniformly in $\T^n$},
\]
\[
d_x\ol\om_\lb^\dt \stackrel{*}{\rightharpoonup} d_xu_\lb,\quad\text{weakly in $L^\infty(\T^n,\R)$}.
\]
For any text function $\phi\in C^2(\T^n,\R)$ with $\phi\geq 0$, due to (H1), we get 
\ben
& &\int_{\T^n} H(x,d_x u_\lb, \lb u_\lb)\phi-u_\lb\Dt(\alpha(x)\phi) dx\\
&=&\lim_{\dt\rightarrow 0_+}\int_{\T^n}H(x,d_x u_\lb, \lb u_\lb)\phi+\langle\partial_p H(x,d_x u_\lb, \lb u_\lb),d_x\ol\om_\lb^\dt-d_xu_\lb\rangle\phi-\ol\om_\lb^\dt\Dt(\alpha(x)\phi) dx\\
&\leq& \lim_{\dt\rightarrow 0_+}\int_{\T^n}H(x,d_x \ol\om_\lb^\dt, \lb u_\lb)\phi-\ol\om_\lb^\dt\Dt(\alpha(x)\phi) dx\leq  \lim_{\dt\rightarrow 0_+}\int_{\T^n}\varpi(\dt)\phi dx=0,
\een
which implies $u_\lb$ is a subsolution of \eqref{eq:hj-e} in the distributional sense. Notice that 
\ben
& &H(x,d_x\om_\lb^\eta(x),\lb u_\lb(x))\\
&=&\alpha(x)\Dt\om_\lb^\eta(x)+\underbrace{\int_{\R^n}\alpha(x+y)\Dt u_\lb(x+y)\zeta^\eta(y)dy-\alpha(x)\Dt \om_\lb^\eta(x)}_{R_2^\eta(x)}+\\
& &\underbrace{H(x,d_x\om_\lb^\eta(x),\lb u_\lb(x))-\int_{\R^n}H(x+y ,d_xu_\lb(x+y),\lb u_\lb(x+y))\zeta^\eta(y) dy}_{R_1^\eta(x)}.
\een
Due to Lemma 3.2 and Lemma 2.4 of \cite{MT}, there holds
\[
|R_2^\eta(x)|\leq C,\quad |R_2^\eta(x)|\leq C\sqrt\eta
\]
for some constant $C>0$. On the other side, due to (H1) and the {\it Jensen's Inequality}, 
\[
R_1^\eta(x)\leq \int_{\R^n} \Big(H(x,d_xu_\lb(x+y),\lb u_\lb(x))-H(x+y ,d_xu_\lb(x+y),\lb u_\lb(x+y))\Big)\zeta^\eta(y) dy
\]
of which for a.e. $y\in B(x,\eta)$,
\ben
& &|H(x,d_xu_\lb(x+y),\lb u_\lb(x))-H(x+y ,d_xu_\lb(x+y),\lb u_\lb(x+y))|\\
&\leq&|H(x+y,d_x u_\lb(x+y),\lb u_\lb(x+y))-H(x+y ,d_xu_\lb(x+y),\lb u_\lb(x))|+\\
& & |H(x+y,d_xu_\lb(x+y),\lb u_\lb(x))-H(x ,d_xu_\lb(x+y),\lb u_\lb(x))|\\
&\leq&\lb \rho^*C_2\eta+ \max_{z\in \T^n}|\partial_x H(z,d_xu_\lb(x+y),\lb u_\lb(x))|\eta\\
&\leq &C_3\eta
\een
in view of \eqref{eq:upper-ishii} for some constant $C_3>0$. Therefore, $|R_1^\eta(x)|\leq C_3\eta$ and $S^\eta(x):=R_1\eta(x)+R_2^\eta(x)$ satisfies the assertion.
\end{proof}

\begin{rmk}\label{rmk:com}
\begin{itemize}
\item[(1)] As an individual interest, the proof of Lemma \ref{lem:com} actually indicates the following byproduct: 

 {\it Any continuous viscosity subsolution of \eqref{eq:hj-e} has to be a continuous subsolution of \eqref{eq:hj-e} in the almost everywhere sense, vice versa.}
 
 Here is the reason: On one side, a viscosity subsolution of \eqref{eq:hj-e} has to be a subsolution in the distributional sense of \eqref{eq:hj-e} can be concluded from above proof, then has to be a subsolution of \eqref{eq:hj-e} in the almost everywhere sense further (due to \eqref{eq:upper-ishii}). On the other side,  suppose $\om_\lb$ is a subsolution of \eqref{eq:hj-e} in the almost everywhere sense, by \eqref{eq:cri-sol-mol} we can get a smooth modification 
 $\om_\lb^\eta$ of $\om_\lb$ for any $\eta>0$. In view of Lemma \ref{lem:com}, $\lim_{\eta\rightarrow 0_+}|\om_\lb^\eta-\om_\lb|=0$ and the stability of viscosity solutions (see \cite{CIL} for instance), $\om_\lb$ has to be a viscosity subsolution of \eqref{eq:hj-e}.
 
 \item[(2)] Notice that the estimate in Lemma \ref{lem:com} also applies to the case $\lb=0$ (although this case has been proved in \cite{MT}), i.e. for any viscosity solution $\om(x)$ of \eqref{eq:hj-e-cri}, the associated $\om^\eta(x)$ given by \eqref{eq:cri-sol-mol} satisfies 
 \[
H(x,d_x\om^\eta(x),0)\leq \alpha(x)\Dt\om_\lb^\eta(x)+S^\eta(x),\quad x\in \T^n
\]
with $
|S^\eta(x)|\leq C,\;\;|S^\eta|_{L^\infty}\leq C\sqrt\eta,\;\;|\eta^2\Dt\om^\eta|\leq C\eta
$ for some constant $C>0$.
 \end{itemize} 
\end{rmk}

\begin{lem}[upper estimate]\label{lem:up}
For any subsequence $\{\lb_i\}_{i\in\N}$ converging to $0$ such that $u_{\lb_i}$ uniformly converges to a solution $\om$ of \eqref{eq:hj-e-cri}, there holds
\beq\label{eq:up}
-\int_{T\T^n}\om (x)\partial_uL(x,v,0) d\mu\leq 0,\quad \forall \,\mu\in\mathcal M'.
\eeq
\end{lem}
\begin{proof} 
Due to Lemma \ref{lem:com}, we denote  
\[
\psi_i^\eta(x):=\int_{\R^n}\eta^{-n}\zeta(\eta^{-1}y)u_{\lb_i}(x+y)dy,\quad x\in M.
\]
 By the convexity of $H$, we have 
\ben
\partial_uL(x,v,0)\lb_i u_{\lb_i}&=&L(x,v,\lb_i u_{\lb_i})-L(x,v,0)-\lb_i u_{\lb_i} Q_{\lb_i}(x,v)\\
&\geq& \langle v,d_x\psi_i^\eta\rangle-H(x, d_x\psi_i^\eta,\lb_i u_{\lb_i})-L(x,v,0)-\lb_i u_{\lb_i} Q_{\lb_i}(x,v)\\
&\geq&\langle v,d_x\psi_i^\eta\rangle-\alpha(x)\Dt \psi_i^\eta(x)-S^\eta(x)-L(x,v,0)-\lb_i u_{\lb_i} Q_{\lb_i}(x,v)
\een
with 
\[
Q_{\lb_i}(x,v):=\int_0^1 \partial_u L\big(x,v,\lb_i(1-\theta) u_{\lb_i}(x)\big)d\theta-\partial_uL(x,v,0).
\]
Integrating both sides of previous inequality by any $\mu\in\mathcal M'$, we get 
\[
\int_{T\T^n} \partial_uL(x,v,0) u_{\lb_i} d\mu\geq -\frac1{\lb_i}\int_{T\T^n}S^\eta d\mu-\int_{T\T^n}u_{\lb_i} Q_{\lb_i}(x,v)d\mu.
\]
Letting $\eta\rightarrow 0_+$ there holds 
\[
-\int_{T\T^n} \partial_uL(x,v,0) u_{\lb_i} d\mu\leq\int_{T\T^n}u_{\lb_i} Q_{\lb_i}(x,v)d\mu .
\]
then taking $i\rightarrow+\infty$ and using the {\it Lebesgue Dominated Convergence Theorem} we get the desired conclusion.
\end{proof}

\begin{lem}[lower estimate]\label{lem:low}
Suppose $\om$ is a viscosity solution of \eqref{eq:hj-e-cri}and $\om^\eta$ is the function given by \eqref{eq:cri-sol-mol}, then 
 for any solution $u_\lb$ of \eqref{eq:hj-e} and $\theta_\lb$ of \eqref{eq:hj-ad} there holds 
\be
u_\lb^\eta(x)-\om^\eta(x)&\geq &-\int_{\T^n}\om^\eta \partial_u H(y, d_xu_\lb^\eta(y),0)\theta_\lb^\eta dy-C\frac\eta\lb \int_{\T^n} \theta_\lb^\eta(y)dy-\frac1\lb\int_{\T^n}S^\eta\theta_\lb^\eta dy\\
& &-\int_{\T^n}  u_\lb^\eta(x) \Big(\int_0^1\Big[\partial_u H(x,d_xu_\lb^\eta,(1-\vartheta)\lb u_\lb^\eta)-\partial_u H(x,d_xu_\lb^\eta,0)\Big]d\vartheta \Big)\theta_\lb^\eta(x)dx\nonumber
\ee
\end{lem}

\begin{proof}
In view of item (2) of Remark \ref{rmk:com}, we have 
\[
H(x, d_x \om^\eta,0)\leq (\alpha(x)+\eta^2)\Dt \om^\eta(x)+C\eta+S^\eta(x).
\]
Subtracting \eqref{eq:hj-e-app} by this inequality, we get 
\ben
(\alpha(x)+\eta^2)\Delta(u_\lb^\eta-\om^\eta)&\leq&H(x,d_xu_\lb^\eta,\lb u_\lb^\eta)-H(x,d_x\om^\eta,0)+C\eta+S^\eta(x)\\
&=&H(x,d_xu_\lb^\eta,\lb u_\lb^\eta)-H(x,d_xu_\lb^\eta,0)+H(x,d_xu_\lb^\eta,0)-H(x,d_x\om^\eta,0)\\
& &+C\eta+S^\eta(x)\\
&\leq& \lb u_\lb^\eta\int_0^1\partial_u H(x,d_xu_\lb^\eta,(1-\vartheta)\lb u_\lb^\eta)d\vartheta\\
& &+\langle\partial_p H(x,d_xu_\lb^\eta,0),d_x(u_\lb^\eta-\om^\eta)\rangle+C\eta+S^\eta(x)\\
&=&  \lb u_\lb^\eta \partial_u H(x,d_xu_\lb^\eta,0)+\langle\partial_p H(x,d_xu_\lb^\eta,0),d_x(u_\lb^\eta-\om^\eta)\rangle+C\eta+S^\eta(x)\\
& &+ \lb u_\lb^\eta \int_0^1\big[\partial_u H(x,d_xu_\lb^\eta,(1-\vartheta)\lb u_\lb^\eta)-\partial_u H(x,d_xu_\lb^\eta,0)\big]d\vartheta
\een
which indicates
\ben
\lb\om^\eta(x)\partial_u H(x,d_xu_\lb^\eta,0)&\geq& (a(x)+\eta^2)\Delta(u_\lb^\eta-\om^\eta)-\langle\partial_p H(x,d_xu_\lb^\eta,0),d_x(u_\lb^\eta-\om^\eta)\rangle\\
& &-\lb (u_\lb^\eta-\om^\eta)\partial_u H(x,d_xu_\lb^\eta,0)-C\eta-S^\eta(x)\\
& &- \lb u_\lb^\eta \int_0^1\big[\partial_u H(x,d_xu_\lb^\eta,(1-\vartheta)\lb u_\lb^\eta)-\partial_u H(x,d_xu_\lb^\eta,0)\big]d\vartheta
\een

Integrating both sides with respect to  the measure $\theta_\lb^\eta(x)dx$, we get  
\ben
& &\int_{\T^n} \om^\eta(x)\partial_u H(x,d_xu_\lb^\eta,0)\theta_\lb^\eta(x) dx\\
&\geq&\big(\om(x_0)-u_\lb(x_0)\big)-C\frac\eta\lb\int_{\T^n} \theta_\lb^\eta(x)dx-\frac1\lb\int_{\T^n}S^\eta(x)\theta_\lb^\eta(x)dx\\
& &-\int_{\T^n}  u_\lb^\eta(x) \Big(\int_0^1\Big[\partial_u H(x,d_xu_\lb^\eta,(1-\vartheta)\lb u_\lb^\eta)-\partial_u H(x,d_xu_\lb^\eta,0)\Big]d\vartheta \Big)\theta_\lb^\eta(x)dx
\een
 Since $x_0\in \T^n$ is freely chosen, rearrange this inequality we get the assertion.
\end{proof}

\bigskip


{\it Proof of Theorem \ref{thm:1}:}  Due to Lemma \ref{lem:up}, any uniform limit of $u_\lb$ along subsequences belongs to $\cS'$, so $\limsup_{\lb\rightarrow 0_+}u_\lb(x)\leq \sup_{\om\in\cS'}\om(x)$; On othe other side, for any $\om\in\cS'$, Lemma \ref{lem:low} indicates 
\ben
& &\liminf_{\lb\rightarrow 0_+}u_\lb(x)\\
&\geq& \om(x)+\liminf_{\lb\rightarrow 0_+}\liminf_{\eta\rightarrow 0_+}\bigg(\int_{\T^n}\theta_\lb^\eta(y)dy\cdot \int_{T\T^n}\om^\eta(y) \partial_u L(y,v,0) d\mu_\lb^\eta(y,v)\\
& &-\int_{\T^n}  |u_\lb^\eta(x)|\cdot\Big|\int_0^1\Big[\partial_u H(x,d_xu_\lb^\eta,(1-\vartheta)\lb u_\lb^\eta)-\partial_u H(x,d_xu_\lb^\eta,0)\Big]d\vartheta \Big|\theta_\lb^\eta(x)dx\bigg)\\
&\geq &\om(x)
\een
since any weak* limit of $\mu_\lb^\eta$ ia contained in $\mathcal M'$. So we get $\sup_{\om\in\cS'}\om(x)\leq \liminf_{\lb\rightarrow 0_+}u_\lb(x)$ and finish the proof. \medskip

\begin{rmk}\label{rmk:main}
\begin{itemize}
\item If additionally we assume 

\quad (H6) For any $R>0$, there exists $B_R>0$ such that 
\[
|\partial_uH(x,p,u)-\partial_u H(x,p,0)|\leq B_R |u|,\quad\forall (x,p)\in T^*\T^n, |u|\leq R.
\]
then the conclusion of Lemma \ref{lem:up} can be generalized to 
\[
-\int_{T\T^n}\om(x)\partial_u L(x,v,0)d\mu\leq 0,\quad\forall \mu\in\mathcal M
\]
for any accumulating function $\om$ of the family $\{u_\lb\}$ as $\lb\rightarrow 0_+$. Furthermore, for Hamiltonians satisfying (H1)-(H6) we can prove 
\beq\label{eq:substitute}\tag{*}
\lim_{\lb\rightarrow 0_+}u_\lb(x)=\sup_{\om\in\cS}\om(x)
\eeq
with  $\cS:=\{\om \text{ is a viscosity solution of }\eqref{eq:hj-e-cri}| \int_{T\T^n}\partial_{u}L(x,v,0)\om d\mu\geq 0,\;\forall \,\mu\in\mathcal M\}$. As already addressed in \cite{MT},  deeper properties about stochastic Mather measures (e.g. Lipschitz graph property and compactness of $\mathcal M$) are still unknown, but important to explore. In view of such a situation, additional assumption like (H6) is inevitable to ensure \eqref{eq:substitute} hold.

\item As is shown in \cite{WYZ}, we can indeed get different solutions of \eqref{eq:hj-e-cri} by choosing different $\partial_u L(x,v,0)$ functions then get different limit of associated $\{u_\lb\}_{\lb>0}$. To illustrate the dynamical differences between these different limit solutions would be also very meaningful in the furture study.
\end{itemize}
\end{rmk}

\bigskip

\appendix

\section{Adjoint equation}\label{a1}

For $\lb>0$, suppose $u:\T^n\rightarrow\R$ is the viscosity solution of 
\[
 \lb\beta(x)u+\langle V(x), d_x u\rangle=(\alpha(x)+\eta^2)\Dt u+\lb f(x),\quad x\in \T^n, 
\]
then $u\geq 0$ as long as $C(\T^n,\R)\ni f,\beta \geq 0$ (due to the comparison principle). As its adjoint equation, there holds
\[
\lb\beta\theta-{\rm div}\big(V(x)\theta\big)=\Dt\big((\alpha+\eta^2)\theta\big)+ \lb \dt_{x_0}
\]
for some $x_0\in \T^n$. As we can see, 
\ben
\int_{\T^n} \lb f\theta dx &=&\int_{\T^n}\Big(\lb \beta(x)u+\langle V(x), d_x u\rangle-(\alpha(x)+\eta^2)\Dt u\Big)\theta(x)dx\\
&=&\int_{\T^n}\Big(\lb \beta\theta-{\rm div}\big(V(x)\theta\big)-\Dt\big((\alpha+\eta^2)\theta\big)\Big)u dx\\
&=&\int_{\T^n}\lb \beta \dt_{x_0}u dx=\lb\beta(x_0)u(x_0)\geq 0
\een
for any $f,\beta\geq 0$. Consequently, $\theta\geq 0$ on $\T^n$ and 
\[
\int_{\T^n} \beta\theta dx=\int_{\T^n}\dt_{x_0}dx=1.
\]
Furthermore, if $\beta>0$, then 
\[
\frac{1}{\max_{\T^n}\beta}\leq\int_{\T^n}\theta dx\leq \frac{1}{\min_{\T^n}\beta}.
\]
Applying previous procedure to \eqref{eq:hj-ad} by taking 
\[
\beta(x)= \partial_u H(x,d_x u_\lb^\eta,0),\quad V(x)=\partial_p H(x,d_xu_\lb^\eta,0),\quad
\]

we instantly get \eqref{eq:meas-est}.\medskip

\noindent{\it Proof of Lemma \ref{lem:visc-est}:} Differentiating both sides of \eqref{eq:hj-e-app} by $x$, then we get 
\ben
& &\partial_x H(x, d_xu_\lb^\eta,\lb u_\lb^\eta)+\partial_p H(x, d_xu_\lb^\eta,\lb u_\lb^\eta)\cdot D^2u_\lb^\eta(x)+\lb \partial_u H(x, d_xu_\lb^\eta,\lb u_\lb^\eta)d_xu_\lb^\eta(x)\\
&=&(\alpha+\eta^2)\Dt (d_xu_\lb^\eta)+d_x\alpha(x)\Dt u_\lb^\eta.
\een
Multiplying previous equality by $d_xu_\lb^\eta$, then we get 
\be\label{eq:2nd-est}
& &\langle\partial_x H(x, d_xu_\lb^\eta,\lb u_\lb^\eta), d_xu_\lb^\eta\rangle+\langle\partial_p H, d_x\psi(x)\rangle+2\lb \partial_u H(x, d_xu_\lb^\eta,\lb u_\lb^\eta)\psi(x)\\
&=&(\alpha+\eta^2)(\Dt\psi-|D^2u_\lb^\eta|^2)+\langle d_x\alpha, d_xu_\lb^\eta\rangle\Dt u_\lb^\eta\nonumber
\ee
where $\psi(x):=\dfrac{|d_xu_\lb^\eta(x)|^2}{2}$. Since 
\beq\label{eq:uni-c1}
|u_\lb^\eta|+|d_xu_\lb^\eta|_{L^\infty}\leq C_4,\quad\forall \lb,\;\eta\in(0,1]
\eeq
for some constant $C_4>0$, there exists a constant $C_5>0$ such that 
\[
|\langle\partial_x H(x, d_xu_\lb^\eta,\lb u_\lb^\eta), d_xu_\lb^\eta\rangle|\leq C_5.
\]
On the other side, 
\be
|\langle d_x\alpha, d_xu_\lb^\eta\rangle \Dt u_\lb^\eta|&\leq& | d_x\alpha |\cdot|d_xu_\lb^\eta|\cdot|\Dt u_\lb^\eta|\\
&\leq& C_4 | d_x\alpha |\cdot|\Dt u_\lb^\eta|=\frac {C_4}\dt\cdot  \dt  | d_x\alpha |\cdot|\Dt u_\lb^\eta|\nonumber\\
&\leq&\frac12 \Big(\frac{C_4^2}{\dt^2}+\dt^2  | d_x\alpha |^2\cdot|\Dt u_\lb^\eta|^2\Big)\nonumber\\
&\leq& \frac12 \Big(\frac{C_4^2}{\dt^2}+\dt^2 C_6\alpha(x)|D^2 u_\lb^\eta|^2\Big)\nonumber
\ee
for some constant $C_6>0$, since $\alpha\geq 0$ then $\sqrt{\alpha}\in {\rm Lip}(\T^n,\R)$ in view of Theorem 5.2.3 of \cite{SV}. Furthermore, previous inequality leads to 
\[
|\langle d_x\alpha, d_xu_\lb^\eta\rangle \Dt u_\lb^\eta|\leq \frac{C_4^2C_6}{2}+\frac12\alpha(x)|D^2 u_\lb^\eta|^2
\]
by taking $\dt^2=1/C_6$. Accordingly, \eqref{eq:2nd-est} implies 
\be\label{eq:2nd-est-2}
& &\langle\partial_p H, d_x\psi(x)\rangle+2\lb \partial_u H(x, d_xu_\lb^\eta,\lb u_\lb^\eta)\psi(x)
-(\alpha+\eta^2)\Dt\psi +\frac{\alpha+\eta^2}{2} |D^2u_\lb^\eta|^2\\
&\leq& C_7:=C_5+\frac{C_4^2C_6}{2}.\nonumber
\ee
Suppose $\theta_\lb^\eta(x)$ is now the solution of the following adjoint equation 
\be\label{eq:ad-app-2}
2\lb \partial_u H(x,d_xu_\lb^\eta,\lb u_\lb^\eta)\theta_\lb^\eta-{\rm div}\Big(\partial_p H(x,d_xu_\lb^\eta,\lb u_\lb^\eta)\theta_\lb^\eta\Big)=\Dt\Big((\alpha+\eta^2)\theta_\lb^\eta\Big)+2\lb\dt_{x_0}
\ee

Integrating both sides of \eqref{eq:2nd-est-2} by $\theta_\lb^\eta(x)dx$ we get 
\ben
\int_{\T^n} (\alpha+\eta^2) |D^2u_\lb^\eta|^2\theta_\lb^\eta(x)dx&\leq& 2\lb |\psi(x_0)|+2C_4\int_{\T^n}\theta_\lb^\eta d x\\
&\leq&\lb C_1^2+\frac{2C_7}{\min_{x\in \T^n}\partial_u H(x,d_xu_\lb^\eta,\lb u_\lb^\eta)}
\een
which further indicates 
\[
\int_{\T^n}  |D^2u_\lb^\eta|^2\theta_\lb^\eta(x)dx\leq \frac{C_8}{\eta^2},\quad \forall\;\eta\in(0,1]
\]
for some constsnt $C_8>0$ due to \eqref{eq:uni-c1}.

Secondly, since $u_\lb^\eta(x)$ is smooth of $\eta\in(0,1]$, so we can take the derivative of \eqref{eq:hj-e-app} with respect to  $\eta$, such that 
\[
\langle \partial_p H(x,d_xu_\lb^\eta,\lb u_\lb^\eta),\partial^2_{x\eta}u_\lb^\eta\rangle +\lb \partial_u H\cdot \partial_\eta u_\lb^\eta=2\eta\Dt u_\lb^\eta+(\alpha+\eta^2)\Dt(\partial_\eta u_\lb^\eta).
\]
Consequently, 
\ben
& &\int_{\T^n}2\lb \partial_u H(x,d_xu_\lb^\eta,\lb u_\lb^\eta)\partial_\eta u_\lb^\eta\theta_\lb^\eta dx\\
&=&\int_{\T^n}\lb \partial_u H(x,d_xu_\lb^\eta,\lb u_\lb^\eta)\partial_\eta u_\lb^\eta\theta_\lb^\eta dx+\int_{\T^n} \big(2\eta\Dt u_\lb^\eta+(\alpha+\eta^2)\Dt(\partial_\eta u_\lb^\eta)-\langle \partial_p H,\partial^2_{x\eta}u_\lb^\eta\rangle\big)\theta_\lb^\eta dx 
\een
which can be further transferred into 
\[
2\eta\int_{\T^n} \Dt u_\lb^\eta\theta_\lb^\eta dx+\lb\int_{\T^n} \partial_u H(x,d_xu_\lb^\eta,\lb u_\lb^\eta)\partial_\eta u_\lb^\eta\theta_\lb^\eta dx=2\lb \int_{\T^n} \dt_{x_0}\partial_\eta u_\lb^\eta dx =2\lb \partial_\eta u_\lb^\eta(x_0).
\]
Since $x_0\in \T^n$ is freely chosen, so we can make $|\partial_\eta u_\lb^\eta(x_0)|=\max_{x\in \T^n}|\partial_\eta u_\lb^\eta(x)|$. If so, 
\ben
2\eta\bigg|\int_{\T^n}\Dt u_\lb^\eta\theta_\lb^\eta dx\bigg|&=&\bigg|2\lb \partial_\eta u_\lb^\eta(x_0)-\lb\int_{\T^n} \partial_u H(x,d_xu_\lb^\eta,\lb u_\lb^\eta)\partial_\eta u_\lb^\eta\theta_\lb^\eta dx\bigg|\\
&\geq &2\lb |\partial_\eta u_\lb^\eta(x_0)|-\lb \int_{\T^n} \max_{x\in \T^n}|\partial_\eta u_\lb^\eta(x)|\partial_u H\cdot \theta_\lb^\eta dx\\
&= &\lb |\partial_\eta u_\lb^\eta(x_0)|.
\een
On the other side, 
\ben
2\eta\bigg|\int_{\T^n}\Dt u_\lb^\eta\theta_\lb^\eta dx\bigg|&\leq&2\eta\int_{\T^n}|\Dt u_\lb^\eta|\theta_\lb^\eta dx\\
&\leq&2\eta\sqrt{\int_{\T^n} |D^2u_\lb^\eta|^2\theta_\lb^\eta dx}\cdot\sqrt{\int_{\T^n}\theta_\lb^\eta dx}\\
&=&2\eta\frac{\sqrt{C_8}}{\eta}\cdot \frac{1}{\min_{x\in \T^n}\partial_u H(x,d_xu_\lb^\eta,\lb u_\lb^\eta)}
\een
due to the {\it H\"older's Inequality}. Combining these two conclusions we get 
\[
|\partial_\eta u_\lb^\eta(x)|\leq \frac{2\sqrt{C_8}}{\lb \min_{x\in \T^n}\partial_u H(x,d_xu_\lb^\eta,\lb u_\lb^\eta)},
\]
then integrate both sides with respect to  $\eta\in(0,1]$ we get 
\[
|u_\lb^\eta-u_\lb|_{L^\infty}\leq C'\frac\eta\lb,\quad \forall \;\lb, \eta\in(0,1]
\]
for some constant $C'>0$. 
%
%
%
%
%

\end{document}